\documentclass[11pt]{amsart}

\newtheorem{theorem}{Theorem}[section]

\newtheorem{cor}[theorem]{Corollary}

\theoremstyle{definition}
\newtheorem{definition}[theorem]{Definition}

\theoremstyle{remark}

\numberwithin{equation}{section}

\begin{document}

\newcommand{\spacing}[1]{\renewcommand{\baselinestretch}{#1}\large\normalsize}
\spacing{1.14}

\title{Flag curvature of invariant $(\alpha,\beta)$-metrics of type $\frac{(\alpha+\beta)^2}{\alpha}$}

\author {H. R. Salimi Moghaddam}

\address{Department of Mathematics, Faculty of  Sciences, University of Isfahan, Isfahan,81746-73441-Iran.} \email{salimi.moghaddam@gmail.com and hr.salimi@sci.ui.ac.ir}


%%\date{\today}

\begin{abstract}
In this paper we study flag curvature of invariant
$(\alpha,\beta)$-metrics of the form
$\frac{(\alpha+\beta)^2}{\alpha}$ on homogeneous spaces and Lie
groups. We give a formula for flag curvature of invariant metrics
of the form $F=\frac{(\alpha+\beta)^2}{\alpha}$ such that $\alpha$
is induced by an invariant Riemannian metric $g$ on the
homogeneous space and the Chern connection of $F$ coincides to the
Levi-Civita connection of $g$. Then some conclusions in the cases
of naturally reductive homogeneous spaces and Lie groups are
given.
\end{abstract}

\maketitle

%%---------------------------INTRODUCTION--------------------------

\section{\textbf{Introduction}}\label{intro}
Finsler geometry is an interesting field in differential geometry
which has found many applications in physics and biology (see
\cite{AnInMa,As}.). One of important quantities which can be used
for characterizing Finsler spaces is flag curvature. The
computation of flag curvature , which is a generalization of the
concept of sectional curvature in Riemannian geometry, is very
difficult. Therefore finding an explicit formula for computing it
can be useful for characterizing Finsler spaces. Also it can help
us to finding new examples of spaces with some curvature properties.\\
Working on a general Finsler space for finding an explicit formula
for the flag curvature is very computational because of
computation in local coordinates. A family of spaces which has
many applications in physics is homogeneous spaces (in particular,
Lie groups) equipped with invariant metrics. The study of
homogeneous spaces (Lie groups) with invariant Riemannian metrics
has been a very hot field in last decades (for example see
\cite{BrFiSpTaWu,KoNo,Mi,No,Pu}.). During recent years, some of
these results extended to Finsler spaces in some special cases
(see \cite{EsSa1,EsSa2,DeHo1,DeHo2,Sa}.). $(\alpha,\beta)$-metrics
are interesting Finsler metrics which have been studied by many
Finsler geometers. $F=\frac{(\alpha+\beta)^2}{\alpha}$ is a
special $(\alpha,\beta)$-metric which has been studied by Z. Shen
and G. C. Yildirim (see \cite{ShYi}.). In this paper we study flag
curvature of these metrics on homogeneous spaces $G/H$ which are
invariant under the action $G$. We suppose that $\alpha$ is
induced by an invariant Riemannian metric $g$ on the homogeneous
space and the Chern connection of $F$ coincides to the Levi-Civita
connection of $g$. Also we study the special cases when $( G/H , g
)$ is naturally reductive or when $H$ is trivial ($H=\{e\}$).

%%---------------------CURVATURE FORMULA--------------------------

\section{\textbf{Preliminaries}}
Let $M$ be a smooth manifold. Suppose that $g$ and $b$ are a
Riemannian metric and  a 1-form on $M$ respectively as follows:
\begin{eqnarray*}
  g&=&g_{ij}dx^i\otimes dx^j \\
  b&=&b_idx^i.
\end{eqnarray*}
In this case we can define a function on $TM$ as follows:
\begin{eqnarray*}
  F(x,y)=\frac{(\alpha(x,y)+\beta(x,y))^2}{\alpha(x,y)},
\end{eqnarray*}
where $\alpha(x,y)=\sqrt{g_{ij}(x)y^iy^j}$ and
$\beta(x,y)=b_i(x)y^i$.\\
It has been shown $F$ is a Finsler metric if and only if for any
$x\in M$, $\|\beta_x\|_{\alpha}<1$, where
\begin{eqnarray*}
  \|\beta_x\|_{\alpha}=\sqrt{b_ib^i}=\sqrt{g^{ij}b_ib_j}.
\end{eqnarray*}
%For more details see [Shen and Yildirim, On a class of projectively flat with constant...]
In a natural way, the Riemannian metric $g$ induces an inner
product on any cotangent space $T^\ast_xM$ such that
$<dx^i(x),dx^j(x)>=g^{ij}(x)$. The induced inner product on
$T^\ast_xM$ induce a linear isomorphism between $T^\ast_xM$ and
$T_xM$. Then the 1-form $b$ corresponds to a vector field
$\tilde{X}$ on $M$ such that
\begin{eqnarray*}
  g(y,\tilde{X}(x))=\beta(x,y).
\end{eqnarray*}
Also we have $\|\beta(x)\|_{\alpha}=\|\tilde{X}(x)\|_{\alpha}$
(for more details see \cite{DeHo2}.).

Therefore we can write the Finsler metric
$F=\frac{(\alpha+\beta)^2}{\alpha}$ as follows:
\begin{eqnarray*}
  F(x,y)=\frac{(\alpha(x,y)+
  g(\tilde{X}(x),y))^2}{\alpha(x,y)},
\end{eqnarray*}
where for any $x\in M$,
$\sqrt{g(\tilde{X}(x),\tilde{X}(x))}=\|\tilde{X}(x)\|_{\alpha}<1$.\\
In this paper we use this representation of $F$.\\ (for more
details about Finsler metrics see \cite{BaChSh,Sh})

\section{\textbf{Flag curvature of invariant metrics of type $\frac{(\alpha+\beta)^2}{\alpha}$ on homogeneous spaces}}

In this section we give an explicit formula for the flag curvature
of invariant $(\alpha,\beta)$-metrics of type
$\frac{(\alpha+\beta)^2}{\alpha}$, where $\alpha$ is induced by an
invariant Riemannian metric $g$ on the homogeneous space and the
Chern connection of $F$ coincides to the Levi-Civita connection of
$g$. For this purpose we need the P\"uttmann's formula for the
curvature tensor of invariant Riemannian metrics on homogeneous
spaces (see \cite{Pu}.).

Let $G$ be a compact Lie group, $H$ a closed subgroup, and $g_0$ a
bi-invariant Riemannian metric on $G$. Assume that $\frak{g}$ and
$\frak{h}$ are the Lie algebras of $G$ and $H$ respectively. The
tangent space of the homogeneous space $G/H$ is given by the
orthogonal compliment $\frak{m}$ of $\frak{h}$ in $\frak{g}$ with
respect to $g_0$. Suppose that $g$ is an invariant metric on
$G/H$. The values of $g_0$ and $g$ at the identity are inner
products on $\frak{g}$ which we denote as $<.,.>_0$ and $<.,.>$.
The inner product $<.,.>$ determines a positive definite
endomorphism $\phi$ of $\frak{g}$ such that $<X,Y>=<\phi X,Y>_0$
for all $X, Y\in\frak{g}$.\\

\begin{theorem}\label{flagcurvature}
Let $G, H, \frak{g}, \frak{h}, g, g_0$ and $\phi$ be as above.
Assume that $\tilde{X}$ is an invariant vector field on $G/H$
which is $g(\tilde{X},\tilde{X})<1$ and $\tilde{X}_H=X$. Suppose
that $F=\frac{(\alpha+\beta)^2}{\alpha}$ is the Finsler metric
arising from $g$ and $\tilde{X}$ such that its Chern connection
coincides to the Levi-Civita connection of $g$. Suppose that
$(P,Y)$ is a flag in $T_H(G/H)$ such that $\{Y,U\}$ is an
orthonormal basis of $P$ with respect to $<.,.>$. Then the flag
curvature of the flag $(P,Y)$ in $T_H(G/H)$ is given by

\begin{eqnarray}\label{Fcurvatureformula}
  K(P,Y)=\frac{6<X,R(U,Y)Y>.<X,U>+<R(U,Y)Y,U>(1-<X,Y>^2)}{(1+<X,Y>)^4(2<X,U>^2-<X,Y>^2+1)},
\end{eqnarray}
where,
\begin{eqnarray}
    <X,R(U,Y)Y>&=&\frac{1}{4}(<[\phi U,Y]+[U,\phi Y],[Y,X]>_0+<[U,Y],[\phi Y,X]+[Y,\phi X]>_0)\nonumber\\
    &&+\frac{3}{4}<[Y,U],[Y,X]_\frak{m}>+\frac{1}{2}<[U,\phi X]+[X,\phi U],\phi^{-1}([Y,\phi Y])>_0\\
    &&-\frac{1}{4}<[U,\phi Y]+[Y,\phi U],\phi^{-1}([Y,\phi X]+[X,\phi Y])>_0,\nonumber
\end{eqnarray}
and
\begin{eqnarray}
  <R(U,Y)Y,U>&=&\frac{1}{2}<[\phi U,Y]+[U,\phi Y],[Y,U]>_0\nonumber \\
  && \ \ \ +\frac{3}{4}<[Y,U],[Y,U]_{\frak{m}}>+<[U,\phi U],\phi^{-1}([Y,\phi Y])>_0 \\
  && \ \ \ -\frac{1}{4}<[U,\phi Y]+[Y,\phi U],\phi^{-1}([Y,\phi U]+[U, \phi Y])>_0.\nonumber
\end{eqnarray}
\end{theorem}

\begin{proof}
The Chern connection of $F$ coincides on the Levi-Civita
connection of $g$ therefore we have $R^F(U,V)W=R^g(U,V)W$, where
$R^F$ and $R^g$ are the curvature tensors of $F$ and $g$,
respectively. Let $R:=R^g=R^F$ be the curvature tensor of $F$ (or
$g$). The flag curvature defined as follows (\cite{Sh}):
\begin{equation}\label{flag}
   K(P,Y)=\frac{g_Y(R(U,Y)Y,U)}{g_Y(Y,Y).g_Y(U,U)-g_Y^2(Y,U)},
\end{equation}
where $g_Y(U,V)=\frac{1}{2}\frac{\partial^2}{\partial s\partial
t}(F^2(Y+sU+tV))|_{s=t=0}$.\\
By using the definition of $g_Y(U,V)$ and some computations for
$F$ we have:
\begin{eqnarray}\label{g_Y}
    g_Y(U,V)&=&\frac{4(\sqrt{g(Y,Y)}+g(X,Y))^3}{g(Y,Y)^{5/2}}\{g(X,V)g(Y,U)-g(Y,V)
    g(X,U)\}\nonumber\\
    &+&\frac{2(\sqrt{g(Y,Y)}+g(X,Y))^2}{g(Y,Y)}\{g(U,V)+g(X,U)g(X,V)\\
    &-&\frac{g(X,Y)g(Y,V)g(Y,U)}{g(Y,Y)^{3/2}}+\frac{1}{\sqrt{g(Y,Y)}}(g(X,U)g(Y,V)\nonumber\\
    &+&g(X,Y)g(U,V)+g(X,V)g(Y,U))\}+\frac{(\sqrt{g(Y,Y)}+g(X,Y))^4}{g(Y,Y)^3}\nonumber\\
    &&\{4g(Y,U)g(Y,V)-g(U,V)g(Y,Y)\}+\frac{4(\sqrt{g(Y,Y)}+g(X,Y))^2}{g(Y,Y)}\nonumber\\
    &&(\frac{g(Y,V)}{\sqrt{g(Y,Y)}}+g(X,V))(\frac{g(Y,U)}{\sqrt{g(Y,Y)}}+g(X,U)-\frac{2g(Y,U)}{\sqrt{g(Y,Y)}}\nonumber\\
    && \ \ \ \ \ \ \ -\frac{2g(X,Y)g(Y,U)}{g(Y,Y)}).\nonumber
\end{eqnarray}
By using (\ref{g_Y}) and this fact that $\{Y,U\}$ is an
orthonormal basis for $P$ with respect to $g$ we have:
\begin{eqnarray}\label{eq1}
% \nonumber to remove numbering (before each equation)
  g_Y(R(U,Y)Y,U)&=&(1+<X,Y>)^2\{2<X,U>.<Y,R(U,Y)Y>.(1-2<X,Y>)\nonumber\\
  &&+6<X,R(U,Y)Y>.<X,U>+<R(U,Y)Y,U>.(1-<X,Y>^2)\},
\end{eqnarray}

and
\begin{eqnarray}\label{eq2}
% \nonumber to remove numbering (before each equation)
  g_Y(Y,Y).g_Y(U,U)-g^2_Y(U,Y)&=&(1+<X,Y>))^6\nonumber\\
  &&(2<X,U>^2-<X,Y>^2+1).
\end{eqnarray}
Also by using P\"uttmann's formula (see \cite{Pu} or \cite{Sa}.)
and some computations we have:
\begin{eqnarray}\label{eq3}
<X,R(U,Y)Y>&=&\frac{1}{4}(<[\phi U,Y]+[U,\phi Y],[Y,X]>_0+<[U,Y],[\phi Y,X]+[Y,\phi X]>_0)\nonumber\\
    &&+\frac{3}{4}<[Y,U],[Y,X]_\frak{m}>+\frac{1}{2}<[U,\phi X]+[X,\phi U],\phi^{-1}([Y,\phi Y])>_0\\
    &&-\frac{1}{4}<[U,\phi Y]+[Y,\phi U],\phi^{-1}([Y,\phi X]+[X,\phi
    Y])>_0,\nonumber
\end{eqnarray}
\begin{eqnarray}\label{eq4}
  <R(U,Y)Y,Y>=0,
\end{eqnarray}
and
\begin{eqnarray}\label{eq5}
  <R(U,Y)Y,U>&=&\frac{1}{2}<[\phi U,Y]+[U,\phi Y],[Y,X]>_0+\frac{3}{4}<[Y,U],[Y,U]_{\frak{m}}>\nonumber\\
             &&+<[U,\phi U],\phi^{-1}([Y,\phi Y])>_0\\
             && -\frac{1}{4}<[U,\phi Y]+[Y,\phi U],\phi^{-1}([Y,\phi U]+[U, \phi Y])>_0.\nonumber
\end{eqnarray}
Substituting the equations (\ref{eq1}), (\ref{eq2}), (\ref{eq3}),
(\ref{eq4}) and (\ref{eq5}) in the equation (\ref{flag}) completes
the proof.

\end{proof}

Now we consider a special case of Riemannian homogeneous spaces
which has been named naturally reductive. In this case the above
formula for the flag curvature reduces to a simpler equation.

\begin{definition}
(See \cite{KoNo}) A homogeneous space $M=G/H$ with a $G-$invariant
indefinite Riemannian metric $g$ is said to be naturally reductive
if it admits an $ad(H)$-invariant decomposition
$\frak{g}=\frak{h}+\frak{m}$ satisfying the condition
\begin{eqnarray*}
 B(X,[Z,Y]_{\frak{m}})+B([Z,X]_{\frak{m}},Y)=0 \hspace{1.5cm}\mbox{for} \ \ \ X, Y, Z \in
 \frak{m},
\end{eqnarray*}
where $B$ is the bilinear form on $\frak{m}$ induced by $\frak{g}$
and $[,]_{\frak{m}}$ is the projection to $\frak{m}$ with respect
to the decomposition $\frak{g}=\frak{h}+\frak{m}$.
\end{definition}

\begin{theorem}
In the previous theorem let $G/H$ be a naturally reductive
homogeneous space. Then the flag curvature of the flag $(P,Y)$ in
$T_H(G/H)$ is given by
\begin{eqnarray*}\label{flagnatural}
  K(P,Y)=\frac{6<X,R(U,Y)Y>.<X,U>+<R(U,Y)Y,U>(1-<X,Y>^2)}{(1+<X,Y>)^4(2<X,U>^2-<X,Y>^2+1)},
\end{eqnarray*}
where,
\begin{eqnarray}
    <X,R(U,Y)Y>&=&\frac{1}{4}<X,[Y,[U,Y]_{\frak{m}}]_{\frak{m}}>+<X,[Y,[U,Y]_{\frak{h}}]>
\end{eqnarray}
and
\begin{eqnarray}
  <R(U,Y)Y,U>&=&\frac{1}{4}<U,[Y,[U,Y]_{\frak{m}}]_{\frak{m}}>+<U,[Y,[U,Y]_{\frak{h}}]>
\end{eqnarray}
Notice that $[,]_{\frak{m}}$ and $[,]_{\frak{h}}$ are the
projections of $[,]$ to $\frak{m}$ and $\frak{h}$ respectively.\\
\end{theorem}

\begin{proof}
By using Proposition 3.4 in \cite{KoNo} (page 202) we have:
\begin{eqnarray*}
 R(U,V)W&=&\frac{1}{4}[U,[V,W]_{\frak{m}}]_\frak{m}-\frac{1}{4}[V,[U,W]_{\frak{m}}]_{\frak{m}}\\
 &-&\frac{1}{2}[[U,V]_{\frak{m}},W]_{\frak{m}}-[[U,V]_{\frak{h}},W]
 \ \ \ \mbox{for}  \ \ \ U,V,W \in \frak{m},
\end{eqnarray*}
hence
\begin{eqnarray*}
R(U,Y)Y=\frac{1}{4}[Y,[U,Y]_{\frak{m}}]_{\frak{m}}+[Y,[U,Y]_{\frak{h}}].
\end{eqnarray*}
Now by substituting the last relation in the formula which is
obtained in theorem \ref{flagcurvature} the proof is completed.
\end{proof}

%\section{\textbf{Invariant metrics of type $\frac{(\alpha+\beta)^2}{\alpha}$ on Lie
%groups}}
As a special case of naturally reductive Riemannian homogeneous
spaces we can consider Lie groups equipped with bi-invariant
Riemannian metrics. Therefore we have the following corollary.

\begin{cor}
Let $G$ be a Lie group, $g$ be a bi-invariant Riemannian metric on
$G$, and $\tilde{X}$ be a left invariant vector field on $G$ such
that $g(\tilde{X},\tilde{X})<1$. Suppose that
$F=\frac{(\alpha+\beta)^2}{\alpha}$ is the Finsler metric arising
from $g$ and $\tilde{X}$ on $G$ such that the Chern connection of
$F$ coincides on the Levi-Civita connection of $g$. Then for the
flag curvature of the flag $P=span\{Y,U\}$, where $\{Y,U\}$ is an
orthonormal basis for $P$ with respect to $g$, we have:
\begin{eqnarray*}\label{flagnatural}
  K(P,Y)=\frac{6<X,[Y,[U,Y]]>.<X,U>+<U,[Y,[U,Y]]>(1-<X,Y>^2)}{4(1+<X,Y>)^4(2<X,U>^2-<X,Y>^2+1)},
\end{eqnarray*}
\end{cor}

\begin{proof}
$g$ is bi-invariant therefore $(G,g)$ is naturally reductive. Now
by using theorem \ref{flagnatural} proof is completed.
\end{proof}

Now we give some results which limit Lie groups that have a
Finsler metric of type described in theorem \ref{flagcurvature}.

\begin{theorem}
There is no left invariant non-Riemannian $(\alpha,\beta)-$metric
of type described in theorem \ref{flagcurvature} on connected Lie
groups with a perfect Lie algebra, that is, a Lie algebra
$\frak{g}$ for which the equation $[\frak{g},\frak{g}]=\frak{g}$
holds.
\end{theorem}

\begin{proof}
If the Chern connection of $F$ coincide on Levi-Civita connection
of the left invariant Riemannian metric $g$ then, $F$ is of
Berwald type. Therefore left invariant vector field $X$ is
parallel with respect to $g$ and by using Lemma 4.3 of
\cite{BrFiSpTaWu}, $g(X,[\frak{g},\frak{g}])=0$. Since $\frak{g}$
is perfect therefore $X$ must be zero.
\end{proof}

\begin{cor}
There is not any left invariant non-Riemannian
$(\alpha,\beta)-$metric of type described in theorem
\ref{flagcurvature} on semisimple connected Lie groups.
\end{cor}

\begin{cor}
If a Lie group $G$ admits a left invariant non-Riemannian
$(\alpha,\beta)-$metric of type described in theorem
\ref{flagcurvature} then for sectional curvature of the Riemannian
metric $g$ we have
\begin{eqnarray*}
  K(X,u)\geq 0
\end{eqnarray*}
for all $u$, where equality holds if and only if $u$ is orthogonal
to the image $[X,\frak{g}]$.
\end{cor}

\begin{proof}
Since $F$ is of Berwald type, $X$ is parallel with respect to $g$.
By using Lemma 4.3 of \cite{BrFiSpTaWu}, $ad(X)$ is skew-adjoint,
therefore by Lemma 1.2 of \cite{Mi} we have $K(X,u)\geq 0$.
\end{proof}

%\begin{theorem}

%\end{theorem}

%\begin{theorem}

%\end{theorem}

%%-------------------- BIBLIOGRAPHY------------------------

\bibliographystyle{amsplain}

\end{document}